\documentclass{amsart}

\usepackage{amssymb}
\usepackage{amsmath}
\usepackage{amsthm}

\newtheorem {theorem}{Theorem}[section]
\newtheorem {prop}[theorem]{Proposition}
\newtheorem {cory}[theorem]{Corollary}
\newtheorem {lemma}[theorem]{Lemma}

\theoremstyle{definition}

\numberwithin{equation}{section}

\newcommand{\mpf}{M}
\newcommand{\vmean}{\mathfrak{M}}

\newcommand{\cd}{{\mathbb{C}}^d}

\newcommand{\de}{\delta}

\newcommand{\sid}{\sigma_d}

\newcommand{\za}{\zeta}
\newcommand{\we}{w}
\newcommand{\wgt}{u}

\newcommand{\hol}{\mathcal{H}ol}

\newcommand{\Dbb}{\mathbb D}
\newcommand{\Tbb}{\mathbb T}

\newcommand{\spd}{\partial B_d}

\newcommand{\bd}{B_d}

\newcommand{\Nbb}{\mathbb N}

\newcommand{\vlm}{\nu_d}

\begin{document}

\title[Integral means of holomorphic functions]{Integral means of holomorphic functions as generic log-convex weights}



\author{Evgueni Doubtsov}


\address{St.~Petersburg Department
of V.A.~Steklov Mathematical Institute, Fontanka 27,
St.~Petersburg 191023, Russia}

\address{Department of Mathematics and Mechanics,
St.~Petersburg State University,
Universitetski pr.~28, St.~Petersburg 198504,
Russia}

\email{dubtsov@pdmi.ras.ru}

\thanks{The author was supported by the Russian Science Foundation (grant No. 14-41-00010).}

\begin{abstract}
Let $\mathcal{H}ol(B_d)$ denote the space of holomorphic functions
on the unit ball $B_d$ of $\mathbb{C}^d$, $d\ge 1$.
Given a log-convex strictly positive weight $w(r)$ on $[0,1)$,
we construct a function $f\in\mathcal{H}ol(B_d)$ such that
the standard integral means $M_p(f, r)$ and $\we(r)$ are equivalent for any $0<p\le\infty$.
Also, we obtain similar results related to volume integral means.
\end{abstract}


\maketitle

\section{Introduction}\label{s_int}

Let $\hol(\bd)$ denote the space of holomorphic functions on the
unit ball $\bd$ of $\cd$, $d\ge 1$.
For $0< p < \infty$ and $f\in\hol(\bd)$, the standard integral means $\mpf_p (f, r)$
are defined as
\[
\mpf_p (f, r) = \left( \int_{\spd} |f(r\za)|^p \, d\sid(\za)\right)^\frac{1}{p}, \quad 0\le r <1,
\]
where $\sid$ denotes the normalized Lebesgue measure on the unit sphere $\spd$.
For $p=\infty$, put
\[
M_\infty (f, r) = \sup\{|f(z)|: |z|=r\}, \quad 0\le r <1.
\]

A function $\we: [0,1) \to (0, +\infty)$ is called a weight if
$\we$ is continuous and non-decreasing.
A weight $\we$ is said to be \textsl{log-convex} if
$\log\we(r)$ is a convex function of $\log r$, $0<r<1$.
It is known that $\mpf_p(f, r)$, $0\le r<1$, is a log-convex weight
for any $f\in \hol(\bd)$, $f(0)\neq 0$, $d\ge 1$, $0<p \le \infty$.
In fact, for $d=1$, this result constitutes the classical Hardy convexity theorem (see \cite{H14}).
The corresponding proofs are extendable to all dimensions $d$, $d\ge 2$
(see, for example \cite[Lemma~1]{XZ11}).

In the present paper, for each $0<p\le \infty$, we show that
the functions $\mpf_p(f, r)$, $f\in \hol(\bd)$, $f(0)\neq 0$, are generic log-convex weights in the sense
of the following equivalence:

Let $u, v: X \to (0, +\infty)$.
We say that $u$ and $v$ are equivalent ($u\asymp v$, in brief)
if there exist constants $C_1, C_2>0$ such that
\[
C_1 u(x) \le v(x) \le C_2 u(x), \quad x\in X.
\]

\begin{theorem}\label{t_lp_gen}
Let $d\ge 1$ and let $\we: [0,1)\to (0, +\infty)$ be a log-convex weight.
There exists $f\in \hol(\bd)$ such that
\[
\mpf_p (f,r) \asymp \we(r),\quad 0\le r <1,
\]
for each $0<p\le \infty$.
\end{theorem}

Also, we consider volume integral means for $0<q<\infty$.
The logarithmic convexity properties for such integral means have been recently investigated
in a series of papers (see, for example, \cite{WXZ15, WZ14, XZ11}).
Applying Theorem~\ref{t_lp_gen}, we obtain, in particular, the following result.

\begin{cory}\label{c_vol_example}
Let $d\ge 1$, $0<q< \infty$ and let $\we: [0,1)\to (0, +\infty)$ be a weight.
The following properties are equivalent:
\begin{itemize}
  \item[(i)]
  $\we(r)$ is equivalent to a log-convex weight on $[0,1)$;
  \item[(ii)]
  there exists $f\in\hol(\bd)$ such that
  \[
  \left( \frac{1}{\vlm(r\bd)} \int_{r\bd} |f(z)|^q\, d\vlm(z) \right)^{\frac{1}{q}} \asymp \we(r), \quad 0< r <1,
  \]
  where $\vlm$ denotes the normalized volume measure on $\bd$.
\end{itemize}
\end{cory}

\subsection*{Organization of the paper}
Section~\ref{s_prf_thmLp} is devoted to the proof of Theorem~\ref{t_lp_gen}.
Corollary~\ref{c_vol_example} and other results related to volume integral means
are obtained in Section~\ref{s_volume}.

\section{Proof of Theorem~\ref{t_lp_gen}}\label{s_prf_thmLp}

Put $\Dbb = B_1$ and $\Tbb = \partial \Dbb$.
For a log-convex weight $\we$ on $[0,1)$,
Theorem~1.2 from \cite{AD15} provides functions $f_1, f_2\in\hol(\Dbb)$
such that $|f_1(z)| + |f_2(z)|\asymp \we(|z|)$, $z\in\Dbb$.
These functions are almost sufficient for a proof of Theorem~\ref{t_lp_gen} with $d=1$.
However, we will need additional technical information contained in \cite{AD15}.
Namely, applying Lemma~2.2 from \cite{AD15} and arguing as in the proof of Theorem~1.2 from \cite{AD15},
we obtain the following lemma.

\begin{lemma}\label{l_blms}
Let $\we$ be a log-convex weight on $[0,1)$.
There exist $a_k>0$, $n_k\in\Nbb$, $k=1,2,\dots$,
and constants $r_0\in (\frac{9}{10}, 1)$, $C_1, C_2 >0$
with the following properties:
\begin{align}
  n_k &< n_{k+1},\quad k=1,2,\dots;
  \label{e_blms_nk} \\
  \sum_{k=1}^\infty a_k r^{n_k} &\le C_1 \we(r),\quad r_0 \le r <1;
  \label{e_blms_up} \\
  |g_1(r\za)| + |g_2(r\za)| &\ge C_2 \we(r),\quad r_0 \le r <1,\ \za\in\Tbb;
  \label{e_blms_low}
\end{align}
where
\[
g_1(z) = \sum_{j=1}^\infty a_{2j-1} z^{n_{2j-1}}, \quad
g_2(z) = \sum_{j=1}^\infty a_{2j} z^{n_{2j}}, \quad z\in \Dbb.
\]
\end{lemma}

\begin{proof}[Proof of Theorem~\ref{t_lp_gen}]
We are given a log-convex weight $\we$ on $[0,1)$.
First, assume that $d=1$.
Let $a_k$ and $n_k$, $k=1,2,\dots$, $g_1$ and $g_2$ be those provided by Lemma~\ref{l_blms}.
By \eqref{e_blms_low},
\[
|g_1(r\za)|^2 + |g_2(r\za)|^2 \ge C_3 \we^2(r),\quad r_0 \le r <1,\ \za\in\Tbb.
\]
Using \eqref{e_blms_nk} and integrating the above inequality with respect to Lebesgue measure $\sigma_1$ on $\Tbb$,
we obtain
\[
\sum_{k=1}^\infty a_k^2 r^{2n_k} \ge C_3 \we^2(r),\quad r_0\le r <1.
\]
Therefore,
\[
1+ \sum_{k=1}^\infty a_k^2 r^{2n_k} \ge C_4 \we^2(r),\quad 0\le r <1.
\]
So, by \eqref{e_blms_nk}, we have
\begin{equation}\label{e_disk_m2_low}
M_2(f,r) \ge \we(r), \quad 0\le r <1,
\end{equation}
where
\[
\sqrt{C_4} f(z) = 1+ \sum_{k=1}^\infty a_k z^{n_k}, \quad z\in\Dbb.
\]
Also, \eqref{e_blms_up} guarantees that
\begin{equation}\label{e_disk_up}
|f(r\za)| \le C_0 \we(r), \quad 0\le r <1,\ \za\in\Tbb.
\end{equation}
Hence, $M_2(f,r)\le M_\infty(f,r) \le C\we(r)$, $0\le r <1$.
Combining these estimates and \eqref{e_disk_m2_low}, we conclude that
$M_2(f,r)\asymp M_\infty(f,r) \asymp \we(r)$.
Thus, $\mpf_p(f,r) \asymp \we(r)$, $0\le r <1$, for any $2\le p \le \infty$.

Also, we claim that $\mpf_p(f,r) \asymp \we(r)$ for any $0 < p < 2$.
Indeed,
\eqref{e_disk_m2_low} and \eqref{e_disk_up} guarantee that
\[
\sigma_1 \left\{\za\in\Tbb: |f(r\za)|\ge \frac{\we(r)}{2}\right\} \ge \frac{1}{2 C_0^2}.
\]
Therefore,
$M_\infty(f,r) \ge \mpf_p(f,r) \ge C_p\we(r)$, $0\le r <1$.
So, the proof of the theorem is finished for $d=1$.

Now, assume that $d\ge 2$.
Let $W_k$, $k=1,2,\dots$, be a Ryll--Wojtaszczyk sequence (see \cite{RW83}).
By definition, $W_k$ is a holomorphic homogeneous polynomial of degree $k$,
$\|W_k\|_{L^\infty(\spd)} =1$ and $\|W_k\|_{L^2(\spd)} \ge \de$ for a constant $\de>0$ which does not depend on $k$.
Put
\[
F(z) = 1+ \sum_{k=1}^\infty a_k W_k(z), \quad z\in \bd.
\]
Clearly, \eqref{e_blms_up} guarantees that
$|F(r\za)| \le C \we(r)$, $0\le r <1$, $\za\in\spd$.
Also, the polynomials $W_k$, $k=1,2,\dots$, are mutually orthogonal in $L^2(\spd)$; hence,
$M_2(F, r) \ge C(\de) \we(r)$, $0\le r<1$.
So, arguing as in the case $d=1$, we conclude that $\mpf_p(F, r) \asymp \we(r)$ for any $0<p\le \infty$, as required.
\end{proof}

As indicated in the introduction, for any $f\in \hol(\bd)$,
the function $\mpf_p(f, r)$ is log-convex;
hence, Theorem~\ref{t_lp_gen} implies the following analog of Corollary~\ref{c_vol_example}.

\begin{cory}\label{c_means}
Let $d\ge 1$, $0<p\le \infty$ and let $\we: [0,1)\to (0, +\infty)$ be a weight.
The following properties are equivalent:
\begin{itemize}
  \item[(i)]
  $\we(r)$ is equivalent to a log-convex weight on $[0,1)$;
  \item[(ii)]
  there exists $f\in\hol(\bd)$ such that
  \[
  \mpf_p(f,r) \asymp \we(r), \quad 0\le r <1.
  \]
\end{itemize}
\end{cory}

\section{Volume integral means}\label{s_volume}
In this section, we consider integral means based on volume integrals.
Recall that $\vlm$ denotes the normalized volume measure on the unit ball $\bd$.
For $f\in \hol(\bd)$, $0<q<\infty$ and
a continuous function $u: [0,1) \to (0, +\infty)$, define
\begin{align*}
  \mpf_{q,u} (f, r)
    &= \left( \frac{1}{r^{2d}} \int_{r\bd} |f(z)|^q u(|z|) \, d\vlm(z)\right)^{\frac{1}{q}}, \quad 0< r <1; \\
  \mpf_{q,u} (f, 0) &= |f(0)| u^{\frac{1}{q}}(0).
\end{align*}

\begin{prop}\label{p_volume}
  Let $0<q<\infty$ and let $\wgt, \we: [0,1)\to (0, +\infty)$ be log-convex weights.
There exists $f\in \hol(\bd)$ such that
\[
 \mpf_{q,\frac{1}{\wgt}} (f,r) \asymp \we(r), \quad 0\le r <1.
\]
\end{prop}
\begin{proof}
  By Theorem~\ref{t_lp_gen} with $p=2$,
  there exist $a_k\ge 0$, $k=0,1,\dots$, such that
  \[
  \we^q(t) \asymp \sum_{k=0}^\infty a_k t^k, \quad 0\le t <1.
  \]
Let
\[
\varphi^q(t) = \sum_{k=0}^\infty (k+2d) a_k t^k, \quad 0\le t <1.
\]
The functions $\varphi^q(t)$ and $\varphi(t)$
are correctly defined log-convex weights on $[0,1)$.
Hence, $\varphi(t) \wgt^{\frac{1}{q}}(t)$ is a log-convex weight as the product of two log-convex weights.
By Theorem~\ref{t_lp_gen}, there exists $f\in \hol(\bd)$ such that
\[
\int_{\spd} |f(t\za)|^q\, d\sid(\za)
\asymp \varphi^q(t) \wgt(t), \quad 0\le t <1,
\]
or, equivalently,
\[
\frac{t^{2d-1}}{\wgt(t)} \int_{\spd} |f(t\za)|^q\, d\sid(\za)
\asymp \sum_{k=0}^\infty (k+2d) a_k t^{k+2d-1}, \quad 0\le t <1.
\]
Representing $\mpf^q_{q,\frac{1}{\wgt}} (f,r)$ in polar coordinates and
integrating the above estimates with respect to $t$, we obtain
\begin{align*}
   \mpf^q_{q,\frac{1}{\wgt}} (f,r) &= \frac{2d}{r^{2d}} \int_0^r \int_{\spd} |f(t\za)|^q\, d\sid(\za) \, \frac{t^{2d-1}}{\wgt(t)} dt \\
  &\asymp \sum_{k=0}^\infty a_k r^k, \\
  &\asymp \we^q(r), \quad 0\le r <1,
\end{align*}
as required.
\end{proof}

Clearly, Proposition~\ref{p_volume} is of special interest if
$\mpf_{q,\frac{1}{u}} (f, r)$ is log-convex or equivalent to a log-convex function for any $f\in\hol(\bd)$.
Also, we have to prove Corollary~\ref{c_vol_example}.
So, assume that $u\equiv 1$ and define
\begin{align*}
  \vmean_q(f,r) &= \left( \frac{1}{\vlm(r\bd)} \int_{r\bd} |f(z)|^q\, d\vlm(z) \right)^{\frac{1}{q}}, \quad 0<r<1, \\
  \vmean_q(f,0) &= |f(0)|,
\end{align*}
where $0< q < \infty$.

\begin{proof}[Proof of Corollary~\ref{c_vol_example}]
By Proposition~\ref{p_volume}, (i) implies (ii).
To prove the reverse implication, assume that $\we(t)$ is a weight on $[0,1)$
and $\we(r) \asymp \vmean_q(f,r)$ for some $f\in\hol(\bd)$, $f(0)\neq 0$.

If $d=1$ and $0<q< \infty$, then $\vmean_q(f,r)$ is log-convex by Theorem~1 from \cite{WXZ15}.
So, (ii) implies (i) for $d=1$.
The function $\vmean_q(f,r)$ is also log-convex if $1\le q <\infty$ and $d\ge 2$.
Indeed, we have
\[
\vmean_q(f,r) = \left( \int_{\bd} |f(rz)|^q\, d\vlm(z) \right)^{\frac{1}{q}}, \quad 0\le r <1.
\]
Thus, Taylor's Banach space method applies (see \cite[Theorem~3.3]{T50}).

Now, assume that $d\ge 2$ and $0<q<1$.
The function $\mpf^q_q(f, t)$ is a log-convex weight.
Hence, by Theorem~\ref{t_lp_gen} with $p=2$, there exist $a_k\ge 0$, $k=0,1, \dots$, such that
\[
\mpf^q_q(f,t) \asymp \sum_{k=0}^\infty a_k t^k, \quad 0\le t <1.
\]
Thus,
\begin{align*}
  \vmean^q_q(f,r) &= \frac{2d}{r^{2d}} \int_0^r \mpf^q_q(f,t) t^{2d-1}\, dt\\
  &\asymp \sum_{k=0}^\infty \frac{a_k}{k+2d} r^k, \quad 0\le r <1.
\end{align*}
In other words, $\vmean_q(f,r)$ is equivalent to a log-convex weight on $[0,1)$.
So, (ii) implies (i) for all $d\ge 1$ and $0<q<\infty$.
The proof of the corollary is finished.
\end{proof}

For $\alpha>0$, Proposition~\ref{p_volume} also applies to the following integral means:
\[
\frac{1}{r^{2d}} \int_{r\bd} |f(z)|^p (1-|z|^2)^{\alpha} \, d\vlm(z), \quad 0\le r <1.
\]
However, in general, the above integral means are not log-convex.

\bibliographystyle{amsplain}

\begin{thebibliography}{1}

\bibitem{AD15}
E.~Abakumov and E.~Doubtsov, \emph{Moduli of holomorphic functions and
  logarithmically convex radial weights}, Bull. Lond. Math. Soc. \textbf{47}
  (2015), no.~3, 519--532.

\bibitem{H14}
G.~H. Hardy, \emph{The mean value of the modulus of an analytic function},
  Proc. London Math. Soc. \textbf{14} (1914), 269--277.

\bibitem{RW83}
J.~Ryll and P.~Wojtaszczyk, \emph{On homogeneous polynomials on a complex
  ball}, Trans. Amer. Math. Soc. \textbf{276} (1983), no.~1, 107--116.

\bibitem{T50}
A.~E. Taylor, \emph{New proofs of some theorems of {H}ardy by {B}anach space
  methods}, Math. Mag. \textbf{23} (1950), 115--124.

\bibitem{WXZ15}
Ch. Wang, J.~Xiao, and K.~Zhu, \emph{Logarithmic convexity of area integral
  means for analytic functions {II}}, J. Aust. Math. Soc. \textbf{98} (2015),
  no.~1, 117--128.

\bibitem{WZ14}
Ch. Wang and K.~Zhu, \emph{Logarithmic convexity of area integral means for
  analytic functions}, Math. Scand. \textbf{114} (2014), no.~1, 149--160.

\bibitem{XZ11}
J.~Xiao and K.~Zhu, \emph{Volume integral means of holomorphic functions},
  Proc. Amer. Math. Soc. \textbf{139} (2011), no.~4, 1455--1465.

\end{thebibliography}

\end{document}